\documentclass[12pt]{amsart}
\usepackage{enumitem} 
\usepackage{amssymb}
\usepackage{mathtools}
\usepackage{amsmath}
\usepackage{unicode-math}
\usepackage{amsthm} 
\usepackage[utf8]{inputenc}
                      {\hspace*{\fill}\nobreak$\Box$\par\medskip}
                       {\hspace*{\fill}\nobreak$\Box$\par\medskip}

\newtheorem{thm}{Theorem}[section]

\newtheorem{remrk}[thm]{Remark}
\newtheorem{defn}[thm]{Definition}

\newtheorem{prop}[thm]{Proposition}
\newtheorem{lemma}[thm]{Lemma}

\begin{document}
\title [Curves in High Degree Plane Pencils with Bounded Degree Components]
{Curves in High Degree Plane Pencils with Bounded Degree Components}
\author[H. Suluyer]%
{Hasan~Suluyer}
\email{hsuluyer@metu.edu.tr}
\address{Department of Mathematics, Middle East Technical University, Çankaya, Ankara, 06800 Turkey}
\begin{abstract}
We study pencils of plane curves of sufficiently large degree $d$ with $d^2$ distinct base points, focusing on their reducible curves whose irreducible components have degree at most $k\geq 2$. Combining techniques from algebraic geometry and combinatorics, we derive an explicit upper bound on the number $m_{d,k}$, the maximal number of reducible curves whose irreducible components have degree at most $k$ in a pencil of curves of degree $d$. We prove that, for sufficiently large $d$, pencils with more than six such curves do not exist. Consequently, under the stronger  assumption $d\geq\frac{7}{2}k(k-1)-2$, we obtain the bound $m_{d,k}\leq 6$, improving the previously known bound for $d \ge 2k$. We also establish restrictions on pencils containing reducible curves consisting of one irreducible component of degree $k$ together with lines, and obtain nonexistence results for certain such pencils.
\end{abstract}

\maketitle

\let\thefootnote\relax

\section{Introduction}
 A generic curve of a pencil of plane curves without fixed components is smooth outside the base locus and irreducible; consequently, it contains only finitely many reducible curves. The study of reducible curves in pencils of algebraic plane curves has attracted considerable attention in algebraic geometry. Given a pencil of plane curves of degree $d$, one is naturally interested in understanding the geometry and arithmetic of the curves which fail to be irreducible or smooth. 
 
 A substantial amount of work has focused on bounding the number of reducible curves in a pencil. In \cite{vistoli1993number}, Vistoli established several restrictions on reducible curves in a pencil. Later, Busé and Chèze studied the total order of reducibility of a pencil of plane curves and proved that the number of reducible members, counted with multiplicities, admits an upper bound $d^2-1$ depending on the degree $d$ of the pencil in \cite{buse2011total}. In particular, the case of completely reducible curves in a pencil, namely curves decomposing into unions of lines only, has already been considered in the literature. Yuzvinsky showed that a pencil of plane curves without fixed components can contain at most four completely reducible curves in \cite{yuzvinsky2009new}. The number of completely reducible curves in pencils of degree $d$ curves plays a special role for the theory of hyperplane arrangements, multinets, resonance varieties, and Ceva pencils. For more details, see \cite{Bartz,stipins,milnorfibr,Falk,pereira2008completely}.

More generally, instead of requiring completely reducibility of curves in a plane pencil of degree $d$, one may consider curves whose irreducible components all have degree at most $k$ for $d>k\geq 2$. For a pencil $P$ of degree $d$, let $m_k:=m_k(P)$ denote the number of such curves in $P$, and let $m_{d,k}:=\max m_k(P),$ where the maximum is taken over pencils of degree $d$. In \cite{cogolludo2021free}, Cogolludo and Libgober studied pencils of degree $d$ plane curves whose reducible members have irreducible components of degree at most $k$. They proved that, for $d\geq2k$, $m_{d,k}\leq 11$ and showed that pencils with $m_{d,k}>6$ form only a finite collection. In this paper, we improve the upper bound for sufficiently large degree $d$ by proving that pencils with $m_{d,k}>6$ do not exist. To achieve this, we adapt and extend the method of \cite{suluyer2026number}, where the bound $m_2\leq 6$ was established via estimates for the Euler characteristic of the surface associated with a pencil with $d^2$ distinct base points.

The aim of this paper is to generalize the approach used in \cite{suluyer2026number} to reducible curves whose irreducible components have degree at most $k$ in order to prove that $m_{d,k}\leq 6$ for sufficiently large $d$ and for every $k\geq2$. Furthermore, we find a restriction on the degree of pencils containing reducible curves consisting of one irreducible component of degree $k$ together with several linear components. Using algebraic geometric and combinatorial arguments, we obtain the following result. 

\begin{thm} (see Theorem \ref{reducible curves with at most degree k irred. main thm}) Let $d>k\ge 2$ be integers. Among all pencils of plane curves of degree $d$ in $\mathbb{CP}^2$ with $d^2$ distinct base points, let $m_{d,k}$ be the maximal number of reducible curves, all of whose irreducible components have degree at most $k$, in a pencil. Then, for each fixed $k$, we have $m_{d,k}\le 6$ for all sufficiently large $d$. In fact, $m_{d,k}\le 6$ for $d\ge \frac{7}{2}k(k-1)-2$.
\end{thm}

We obtain the asymptotic bound $m_{d,k} \leq 6$ for each fixed $k\geq 2$ and sufficiently large $d$, showing that it coincides with the known uniform bound for $k=2$.

The paper is organized as follows: In Section 2, we recall factorizations in pencils of plane curves and fibrations obtained from pencils. In Section 3, we obtain a bound on the number $m_{d,k}$ of curves in a pencil whose irreducible components have degree at most $k$ by estimating the Euler characteristic of the associated surface for sufficiently large $d$. Finally, in Section 4, we derive restrictions on pencils containing reducible curves consisting of one irreducible component of degree $k$ together
with several linear components. As a consequence, we obtain nonexistence results for certain pencils of this type.
\enlargethispage{3\baselineskip}
\section{Preliminaries}
\subsection{Factorizations in Pencils of Plane Curves}
Let $C_1$ and $C_2$ be two distinct plane curves of the same degree $d$ defined by homogeneous polynomials $f,g \in \mathbb{C}[x_0,x_1,x_2]$, assumed to have no common factor. The \textit{pencil of plane curves} generated by $C_1$ and $C_2$  in $\mathbb{CP}^2$ is given by
\[
P=\left\{\lambda f(x_0,x_1,x_2)+\mu g(x_0,x_1,x_2)=0 \;\middle|\; [\lambda:\mu]\in \mathbb{CP}^1 \right\}.
\]

All curves in the pencil share a common set of intersection points $\mathcal{X}$, called the \textit{base locus} of the pencil. By Bézout’s theorem, the number of distinct points in $\mathcal{X}$ is less than or equal to $d^2$. Moreover, equality holds when $C_1$ and $C_2$ intersect transversely; that is, they are smooth at the intersection points and have distinct tangent directions at each intersection point.

Throughout the paper, assume that $P$ is a pencil of degree $d$ plane curves intersecting transversely at $d^2$ distinct points.

 For $2\le k<d$, let $m_{k}$ be the number of curves $C_i:\{h_i(x_0,x_1,x_2)=0\}$ in a pencil of degree $d$ curves whose irreducible components have degree at most $k$ for $i \in \{1,\cdots,m_{k}\}$. In other words, for $m_{k}$ different values of $[ \lambda_i : \mu_i] \in \mathbb{CP}^1$, $h_i=\lambda_i f+\mu_i g$ is the product of irreducible polynomials of degree at most $k$.
 \subsection{Natural fibration over $\mathbb{C P}^1$} Let $P$ be a pencil of degree $d$ plane curves with $m$ curves whose irreducible components have degree at most $k$ in $\mathbb{C P}^2$ and let $\mathcal{X}$ denote its base locus. By blowing up $\mathbb{C P}^2$ at the points of $\mathcal{X}$, we obtain a surface $S$. Say $\psi: S  \rightarrow \mathbb{CP}^2$ is the blow-up map with  the exceptional curves $E_1, \ldots, E_{d^2}$. Then, the pencil $P$ induces a natural morphism
$$
\begin{aligned}
\varphi: S & \rightarrow \mathbb{C P}^1 \\
p \in\{\lambda f+\mu g=0\} & \mapsto [ \lambda : \mu]
\end{aligned}
$$
whose fibers are precisely the strict transforms of the curves in the pencil. This morphism defines a fibration over $\mathbb{C P}^1$, with a smooth generic fiber. This fibration has $m$ fibers $W_i$, each of which is the strict transform of a curve in $P$ with irreducible components of degree at most $k$. Each of these fibers is called a \textit{special fiber}. Thus, each such curve $C_i$ gives rise to a special fiber $W_i$ of $\varphi$. Suppose that the special fibers $W_i$ are over $p_1, p_2, \ldots, p_m \in \mathbb{C P}^1$. Then, each exceptional curve $E_i$ intersects exactly one irreducible component of $W_j=\varphi^{-1}(p_j)$ for any $j$.
 
\section{A Bound on $m_{d,k}$ for High-Degree Plane Pencils}

Let $m_{d,k}$ denote the maximal number of curves that factor into components of degree at most $k$ in a pencil of curves of degree $d$. To derive an upper bound for $m_{d,k}$, we estimate the Euler characteristic $e(S)$ of the fibered surface $S$ by separately computing the contributions of the smooth and special fibers of $\varphi$. The following lemma establishes an appropriate lower bound for the Euler characteristics of the special fibers.
\begin{lemma}
\label{characteristic bounds}
   Let $P$ be a pencil of degree $d$ curves with $d^2$ distinct base points.  For any special fiber $W$ of $\varphi$,
   
   $$ e(W) \geq \frac{d(5-d)}{2}- \displaystyle\sum_{n=2}^k {n\choose 2} \alpha_n $$ 
   where $ \alpha_n$ denotes the number of degree $n$ irreducible components of the reducible curve with degree of components at most $k$ in $P$ associated to $W$ .
\end{lemma}
\begin{proof}
     Let $W$ be the special fiber associated with a curve $C$ in $P$ whose irreducible components have degree at most $k$. For each integer $n$ with $0<n\leq k$, let $\alpha_n$ denote the number of irreducible components of the curve $C$ of degree $n$. Since the total degree of $C$ is $d$, we have $\sum_{n=1}^k n\alpha_n =d$. So, $\alpha_1=d-\sum_{n=2}^k n\alpha_n$. As the Euler characteristic of strict transforms of lines in $\mathbb{CP}^2$ is 2, and that of irreducible curves of degree $n$ is at least $3n-n^2$ for $1<n\leq k$, we obtain
  \begin{eqnarray} e(W)&\geq& 2\left(d-\sum_{n=2}^{k}n\alpha_n\right)
+\sum_{n=2}^{k}(3n-n^2)\alpha_n-\displaystyle\sum_{p \in W} (r_p-1)  \nonumber \\
&=& \label{lower bound for e(W) with r_p}
  2d+\sum_{n=2}^{k}(n-n^2)\alpha_n-\displaystyle\sum_{p \in W} (r_p-1) 
  \end{eqnarray}
  where $r_p$ denotes the number of irreducible components of $C$ passing through $p$, with $r_p\geq 1$ for $p\in W$.
  
For a fixed $k$ and fixed integers $\alpha_1,\ldots,\alpha_k$, let $W_{\min}$ denote a fiber whose associated curve decomposes into $\alpha_n$ irreducible curves of degree $n$ for any $n\le k$, such that $e(W_{\min})$ attains its minimum among all such fibers with this decomposition type. In order to find  $e(W_{\min})$, we need to find the maximum value of $\sum_{p \in W} (r_p - 1)$. Each point $p \in W$ with $r_p = r > 2$ arises as a common intersection point of strict transforms of $r$ irreducible components of $C$, and the contribution of the point $p\in W$ to this sum is $r-1$. In contrast, if these $r$ components intersected only transversely, their intersections would consist of $r\choose 2$ nodes, i.e. ordinary double points. For $r>2$, we have
$$r-1 < {r\choose 2}(2-1)=(r-1)\frac{r}{2}.$$ Consequently, for each $p\in W$, we have $r_p-1\leq \binom{r_p}{2}$, with strict inequality whenever $r_p>2$. Hence,
$$\sum_{p\in W}(r_p-1)
\leq
\sum_{p\in W}\binom{r_p}{2}.$$
 
Suppose that every point $p\in W$ with $r_p=r>1$ is a transverse intersection of the strict transforms of $r$ irreducible components of $C$. Then, $\sum_{p\in W}\binom{r_p}{2}=|I|$, where $I$ denotes the set of intersection points of the $\sum_{n=1}^{k}\alpha_n$ irreducible components of $C$ \textit{in general position}, that is, every intersection point is a transverse intersection of exactly two components. On the other hand, if there exists a point $p\in W$ at which two or more components share a common tangent direction, then $\sum_{p \in W} {r_p\choose 2} < |I|$. Consequently, $$\sum_{p \in W} (r_p-1)\leq \sum_{p \in W} {r_p\choose 2}\leq |I|. $$ 
Hence, the maximum value of $\sum_{p \in W} (r_p-1)$
is $|I|$, and it is attained when $W=W_{\min}$, where $W_{\min}$ consists of strict transforms of $\sum_{n=1}^k \alpha_n$ irreducible components in general position. In this case, 
 \begin{equation}
  \label{nodes when special fiber is in general position}
     |I|= {d \choose 2}-\sum_{n=2}^{k}\binom{n}{2}\alpha_n.
  \end{equation}
Indeed, ${d \choose 2}$ corresponds to the case where all components are lines, and replacing $n$ lines by a degree $n$ curve decreases the number of nodes by $\binom{n}{2}$.
   Thus, 
   \begin{eqnarray*}
        e(W_{\min}) &\geq& 2d+\sum_{n=2}^{k}(n-n^2)\alpha_n-
\left(\binom{d}{2}-\sum_{n=2}^{k}\binom{n}{2}\alpha_n\right)\\
&=& \dfrac{d(5-d)}{2}-\sum_{n=2}^{k}\binom{n}{2}\alpha_n. 
 \end{eqnarray*}\qedhere \end{proof}
 
We establish an upper bound for $m_{d,k}$  for sufficiently large degree pencils
 by analyzing the geometry of the pencil and the contributions of its special fibers. In particular, we estimate the Euler characteristic of the surface $S$ and combine this with bounds on the Euler characteristics of singular fibers.

\begin{thm} \label{reducible curves with at most degree k irred. main thm}
 Let $d>k\ge 2$ be integers. Among all pencils of plane curves of degree $d$ in $\mathbb{CP}^2$ with $d^2$ distinct base points, let $m_{d,k}$ be the maximal number of reducible curves, all of whose irreducible components have degree at most $k$, in a pencil. Then, for each fixed $k$, we have $m_{d,k}\le 6$ for all sufficiently large $d$. In fact, $m_{d,k}\le 6$ for $d\ge \frac{7}{2}k(k-1)-2$.

\end{thm}

\begin{proof}
  Let $P$ be a pencil with $m_{d,k}$ such reducible curves, let $S$ be the
blow-up of $\mathbb{CP}^2$ at its base points and $\varphi\colon S\to\mathbb{CP}^1$. Let us estimate $e(S)=3+d^2$ by using the Euler characteristics of special fibers $W_i$ of $S$ for $i=1,2,\cdots,m_{d,k}$. Since the generic fiber is a smooth plane curve of degree $d$, its Euler characteristic is $3d-d^2$. Moreover, this Euler characteristic is always less than that of any singular fiber. Thus, by Lemma \ref{characteristic bounds}, we obtain

\begin{eqnarray}
\label{ineq: upper bound for sum e(W_i)}
e(S)=3+d^2&\geq&  (2-m_{d,k})(3d-d^2)+\displaystyle\sum_{i=1}^{m_{d,k}} e(W_i)\\
&\geq& (2-m_{d,k})(3d-d^2)+m_{d,k}\dfrac{d(5-d)}{2}-\sum_{i=1}^{m_{d,k}}\overline{\alpha_{i,k}} \nonumber
\end{eqnarray}
where $\overline{\alpha_{i,k}}
=\sum_{n=2}^{k}\binom{n}{2}\alpha_{i,n}$ and $\alpha_{i,n}$ denotes the number of degree $n$ irreducible components of the curve in $P$ associated with $W_i$.
Hence,
$$ 3+d^2-6d+2d^2
\geq m_{d,k}\left(\frac{d(5-d)}{2}-(3d-d^2)
\right)
-
\sum_{i=1}^{m_{d,k}}\overline{\alpha_{i,k}}.$$

Therefore,

\begin{equation}
   \label{ineq: crowded} 
3(d-1)^2
\geq
m_{d,k}\frac{d^2-d}{2}
-
\sum_{i=1}^{m_{d,k}}\overline{\alpha_{i,k}}.\end{equation}

Since $\alpha_{i,n}\leq \frac{d}{n}$ for any $i$,

$$\overline{\alpha_{i,k}}
\leq
\sum_{n=2}^{k}\binom{n}{2}\frac{d}{n}
=\frac{dk(k-1)}{4}.$$

Hence, inequality \eqref{ineq: crowded} becomes
\begin{eqnarray*}
3(d-1)^2 &\geq&m_{d,k}\frac{d^2-d}{2}-m_{d,k}\frac{dk(k-1)}{4}\\
&=&m_{d,k}\frac{d(2d-k(k-1)-2)}{4}.
\end{eqnarray*}
Assume that $2d>k(k-1)+2$. Then
$$\frac{12(d-1)^2}{d(2d-k(k-1)-2)} \geq m_{d,k}.$$ Our goal is to show that $7>m_{d,k}$ for sufficiently large $d$. So, it suffices to find a lower bound on $d$ such that
$$7>\frac{12(d-1)^2}{d(2d-k(k-1)-2)}.$$ This is equivalent to

$$12(d-1)^2 < 14d^2 - 7k(k-1)d - 14d.$$
Expanding both sides yields

$$12d^2 - 24d + 12
<
14d^2 - 7k(k-1)d - 14d.$$
Therefore,
\begin{equation}
\label{ineq for suff. d}
d^2 + d\left(5-\frac{7}{2}k(k-1)\right) - 6 > 0.
\end{equation}

Let $b = 5-\frac{7}{2}k(k-1).$ It is clear that $b<0$ for all $k\geq 2$.  Then, inequality \eqref{ineq for suff. d} holds if  $d>d_+:=\dfrac{-b+\sqrt{b^2+24}}{2}.$
Moreover, $\sqrt{a^2+b^2}\leq|a|+|b|$, so
$$d_+
=
\frac{-b+\sqrt{b^2+24}}{2}
\leq
\frac{-b-b+\sqrt{24}}{2}
<
-b+\frac52.$$

Thus,
$$d_+
<
\frac{7}{2}k(k-1)-5+\frac52
<
\frac{7}{2}k(k-1)-2
.$$

Therefore, inequality \eqref{ineq for suff. d} holds for $\frac{7}{2}k(k-1)-2$ as well. \end{proof}

 Ruppert constructed in \cite{ruppert1986reduzibilitat} a pencil of degree $d$ plane curves having exactly $3(d-1)$ reducible curves for any $d>2$. In this example, each reducible curve consists of a line together with an irreducible curve of degree $(d-1)$. So, $m_{d-1}=3(d-1)$, which is greater than 6. This does not contradict the above theorem, since in Ruppert’s example we have $k=d-1$; thus, the condition $d\ge\frac{7k(k-1)}{2}-2$ becomes $d\ge\frac{7(d-1)(d-2)}{2}-2$, which is not satisfied for $d\ge 3$.
\begin{defn}
    A curve in a pencil of degree $d$ plane curves is called a pencil of $d$ lines if it is the union of $d$ distinct lines meeting at a single point.
\end{defn}
The following proposition shows how the existence of pencils of $d$ lines affects the bound on the number of curves whose irreducible components have degree at most $k$ for sufficiently large degree $d$.

  \begin{prop} \label{prop: reducible curves with at most degree k irred. main thm p_k} Let $k\geq 2$ and let $d$ be sufficiently large. Let $P$ be a pencil of plane curves of degree $d$ in $\mathbb{CP}^2$ with $d^2$ distinct base points. Let $m_k(P)$ denote the number of reducible curves in $P$ whose irreducible components all have
degree at most $k$, and let $p$ denote the number of those curves in $P$ that
are pencils of $d$ lines. Then $m_{k}(P)\leq 6- p$. In particular, $p\le 3$.
  \end{prop}
\begin{proof}
   Write $m_k:=m_k(P)$. The Euler characteristic of a pencil of $d$ lines is $d+1$:  it is a union of $d$ copies of $\mathbb{CP}^1$
glued at a single point. Without loss of generality, $e(W_i)=d+1$ for $i=1,2,\cdots,p$ where the $W_i$ are special fibers of the fibration associated with the pencil $P$. So, inequality \eqref{ineq: upper bound for sum e(W_i)} becomes
    
  \begin{eqnarray}
\nonumber
3+d^2&\geq&  (2-m_{k})(3d-d^2)+p(d+1)+\displaystyle\sum_{i=p+1}^{m_{k}} e(W_i)\\
&\geq& (2-m_{k})(3d-d^2)+p(d+1)+(m_{k}-p)\dfrac{d(5-d)}{2}-\sum_{i=p+1}^{m_{k}}\overline{\alpha_{i,k}}. \label{ineq: upper bound for sum e(W_i) with p} 
\end{eqnarray}
Since $\alpha_{i,n}\leq \frac{d}{n}$ for any $i$, we have $\overline{\alpha_{i,k}}
\leq
\sum_{n=2}^{k}\binom{n}{2}\frac{d}{n}
=\frac{dk(k-1)}{4}.$ Substituting this bound into inequality \eqref{ineq: upper bound for sum e(W_i) with p}, we obtain
$$3+d^2 \geq  (2-m_{k})(3d-d^2)+p(d+1)+(m_{k}-p)\dfrac{d(5-d)}{2}-(m_{k}-p)\frac{dk(k-1)}{4}.$$
Therefore,
$$m_{k} \leq \dfrac{12(d-1)^2}{d(2d-k(k-1)-2)}- p\dfrac{2(d-1)(d-2)+dk(k-1)}{d(2d-k(k-1)-2)}$$ provided that $2d > k(k-1)+2$. Taking the limit as $d \to \infty$, we deduce  $m_{k}\leq 6- p$. Since  $p\leq m_{k}$, we obtain $p \leq m_{k}\le 6-p.$ Hence $p\leq 3$.
\end{proof}

\section{Non-existence of Pencils in the case $m_{k}(P)=6$ }
 
We now derive a restriction on the degree of pencils containing reducible curves of a prescribed type. More precisely, we consider pencils with reducible curves consisting of one irreducible component of degree $k$ together with several linear components. The following theorem shows that the existence of six such reducible curves in a pencil imposes a strong upper bound on the degree $d$ of the pencil.

\begin{thm}
    \label{nonexistent pencils}  Let $P$ be a pencil of plane curves of degree $d\ge 3$ in $\mathbb{CP}^2$ with $d^2$ distinct base points. If $P$ contains six reducible curves whose irreducible components consist of one degree $k\geq 2$ curve together with $d-k$ lines, then $P$ does not exist if $d>k(k-1)+1$. 
 
\end{thm}

 \begin{proof} 
Assume, to the contrary, that there exists such a pencil $P$. Let $S$ be the fibered surface associated with the pencil $P$ of degree $d> k(k-1)+1$ plane curves for $k\geq2$. We compute $e(S)=3+d^2$ by comparing the Euler characteristics of the special fibers $W_i$ of $S$. Since the generic fiber is a smooth plane curve of degree $d$, its Euler characteristic equals $ 3d-d^2$, which is strictly smaller than the Euler characteristic of any singular fiber. Then, we obtain
\begin{equation}
\label{ineq: upper bound for sum e(W_i) nonexis}
e(S)=3+d^2 \geq  (2-m_{k})(3d-d^2)+\displaystyle\sum_{i=1}^{m_{k}} e(W_i)
\end{equation} where $W_i$  is the special fiber corresponding to the curve $C_i$ whose irreducible components are one degree $k$ curve and $d-k$ lines. By Lemma \ref{characteristic bounds},  $$ e(W_i) \geq \frac{d(5-d)}{2}- \displaystyle\sum_{n=2}^{k} {n\choose 2} \alpha_{i,n} $$ 
   where $ \alpha_{i,n}$ denotes the number of degree $n$ irreducible components of the reducible curve  in $P$ with degree of components at most $k$ associated with $W_i$ for each $i$. Since  $ \alpha_{i,1}=d-k,\alpha_{i,k}=1 $ and $\alpha_{i,n}=0$ for all other values of $n$, we have $e(W_i) \geq \frac{d(5-d)}{2}- {k\choose 2} $. Thus,  
\begin{equation}
  \label{ineq:lower bound for sum e(W_i) nonexist}  
    \displaystyle\sum_{i=1}^{m_{k}} e(W_i) \geq m_{k}\frac{d(5-d)}{2}-m_{k}{k\choose 2}.
\end{equation}
Using inequalities \eqref{ineq: upper bound for sum e(W_i) nonexis} and \eqref{ineq:lower bound for sum e(W_i) nonexist}, we obtain
$3+d^2- (2-m_{k})(3d-d^2)\geq m_{k}\frac{d(5-d)}{2}-m_{k}{k\choose 2}$. Then,
\begin{eqnarray*}
3(d-1)^2 \geq m_{k}\frac{d^2-d-k(k-1)}{2}.
\end{eqnarray*}
In the case $m_{k}=6$,  this inequality  yields $ 1+k(k-1) \geq d$, contradicting the assumption $d>1+k(k-1)$.\end{proof}
In fact, the same computation shows that such a pencil contains at most five such curves.
  \begin{remrk}
      When $k=2$, there is no such pencil for $d>3$ by Theorem \ref{nonexistent pencils}. However, such a pencil exists when $k=2$ and $d=3$. Ruppert’s example from \cite{ruppert1986reduzibilitat} provides an instance of the case $m_{2}=6$ with $d=3$. It is given by a general pencil in the net of curves generated by the polynomials $x(y^2-z^2)$, $y(x^2-z^2)$ and $z(x^2 - y^2)$. \end{remrk}


\bibliographystyle{abbrv} 
\bibliography{Pencils.bib}
\end{document}